\documentclass[a4paper,12pt,reqno]{amsart}
 \pdfoutput=1 
\usepackage{amsmath,amsthm,amssymb}
\usepackage{hyperref,cleveref}
\usepackage{mathrsfs}
\usepackage{graphicx}
\usepackage{comment}

\allowdisplaybreaks[1]

\title{Real non-attractive fixed point conjecture}
\newtheorem{definition}{Definition}
\newtheorem{example}{Example}
\newtheorem{thm}{Theorem}

\newtheorem{lemma}{Lemma}

\newtheorem*{remark}{Remark}

\author[M. Vaseem]{Mohd. Vaseem }
\address{Department of Mathematics, University of Delhi,
Delhi--110 007, India}
\email{vasli0551@gmail.com}

\begin{document}
\title[Real non-attractive fixed point conjecture for complex harmonic functions]{Real non-attractive fixed point conjecture for complex harmonic functions}
\begin{abstract}
We prove the real non-attractive fixed point conjecture for complex polynomial and rational harmonic functions. A harmonic function $f=h+\overline{g}$ is polynomial (rational) if both $h$ and $g$ are polynomials (rational functions) of degree at least 2. We show that every such function with a super-attracting fixed point has a $\mathfrak{h}$-fixed point  $\zeta=\mu+\overline{\omega}$ such that the real parts of its multipliers satisfy $\text{Re}(\partial_z h(\mu)) \geq 1$ and $\text{Re}(\partial_z g(\omega)) \geq 1$. For polynomial harmonic functions, this holds even without super-attracting conditions. We provide explicit examples, visualizations, and discuss problem for transcendental harmonic functions.
\end{abstract}
%\subjclass[2020]{30D05, 37F10}
\maketitle
\section{Introduction and Preliminaries}
  The real non-attractive fixed point conjecture, posed by Coelho and Kalantari \cite{How many real attractive fixed points can a polynomial have?}, asks whether every polynomial $P(z)$ of degree at least 2 has a fixed point $z_0$ such that the real part of its multiplier is greater than or equal to $1$, i.e.,  with $\text{Re}(P'(z_0)) \geq 1$. This was affirmatively resolved for polynomials and rational functions with super-attracting fixed points in \cite{The real non-attractive fixed point conjecture and beyond}. We extend this result to complex harmonic functions, defined as $f = h + \overline{g}$, where $h$ and $g$ are analytic, and $\overline{g}$ denotes the complex conjugate of $g(z)$. We focus on polynomial harmonic functions (both $h, g$ polynomials of degree $\geq 2$) and rational harmonic functions (both $h, g$ rational of degree $\geq 2$).

Harmonic functions, prevalent in fields like fluid dynamics and geometric function theory \cite{Some questions about complex harmonic functions}, introduce non-holomorphic dynamics via $\mathfrak{h}$-fixed point, distinct from traditional fixed points. Our main results (Theorems \ref{T1} and \ref{T2}) establish that every rational harmonic function with a super-attracting fixed point, and every polynomial harmonic function has a $\mathfrak{h}$-fixed point satisfying the conjecture. Theorem \ref{T3} examines the quadratic family $P(z) = z^2 + c + \overline{z^2 + c}$, identifying conditions for multipliers with real part exactly 1. Examples and visualizations clarify the dynamics, and we propose problem for transcendental harmonic functions.

 \textbf{Setting and Notations:}
\begin{enumerate}
\item[•] If the point at infinity is a fixed point of the rational function $R$ that is  $R(\infty)=\infty$ then the multiplier of this fixed point $\infty$ is defined as $h^\prime(0)$ where the function $h(z)$ is defined as follows: $h(z)=\dfrac{1}{R\left( \frac{1}{z}\right)}$.
\item[•] The local iterative behaviour of any function is controlled by the  multiplier of the function at this fixed point.
\end{enumerate}

If $R$ is a rational function then a fixed point $z_{0}$ of $R$ is nothing but a root of the following function $R(z)-z$. Here the local degree of this root `$m$' is most important. \\
 	The local degree $m$ at this fixed point $z_{0}$ of the function $R(z)-z$ is the natural number $m$ such that the map $R(z)-z$, around the fixed point $\hat{z}$ behaves as the map $z\rightarrow z^{m}$ around the origin. This number $m$ is known as the multiplicity of the fixed point $z_{0}$.
 	 
 	\begin{definition}\label{D1}\cite{The real non-attractive fixed point conjecture and beyond} A fixed $z_{0}$ point of a rational map 
 	$R$ is said to be \textbf{weakly repelling} if its multiplier $\lambda=R^{\prime}(z_{0})$ at $z_{0}$ satisfies $\mid\lambda\mid>1$ or $\lambda=1$.
  \end{definition}

 	\begin{definition}\label{D1}\label{def:mul}\cite{The real non-attractive fixed point conjecture and beyond} Let $R$ be any rational function having a fixed point $\hat{z}$ and Taylor series expansion of the function $R(z)-z$ about the fixed point $\hat{z}$ be,
 	\begin{center}	 $b_{m}(z-\hat{z})^{m}+b_{m+1}(z-\hat{z})^{m+1}+\cdots$ with $b_{m}\neq 0$
 	\end{center}
 	Then the multiplicity of this fixed point $\hat{z}$ of  the rational function $R$ is a unique natural number `$m$'. This fixed point $\hat{z}$ is called  simple if $m=1$. But if $m\neq 1$, it is called a multiple fixed point.
    \end{definition}
 	
 	\begin{remark}
 	    
 	 The multiplicity $m$ of a fixed point $z_{0}$ of a rational function $R$ is nothing but the multiplicity of fixed point $z_{0}$ when it is considered as a root of the function $R(z)-z$.\\
 	The following lemma gives a relation between the multiplier and the multiplicity of a fixed point.
\end{remark}
 
\begin{lemma}\cite{The real non-attractive fixed point conjecture and beyond} The necessary and sufficient condition for a fixed point of a rational function to be multiple is that its multiplier is 1.
 \end{lemma}

\begin{remark}
    
     If a fixed point is attractive or repelling or irrationally indifferent then it is always simple. But if this fixed point is rationally indifferent, then it is simple only in the condition when the multiplier is different from 1.
\end{remark}

\begin{example}
    
$P(z)=iz+z^{2}$, $0$ is fixed point of $P(z)$.

\begin{align*}
P^{\prime}(z)&=i+2z\\
\lambda&=P^{\prime}(0)=i\\
\vert\lambda\vert&=\vert P^{\prime}(0)\vert=1\\
\lambda^{4}&=1
\end{align*}
Here $0$ is simple rationally indifferent fixed points of $P(z)$.\\

\end{example}
To understand the nature of all the fixed points together, firstly it is
necessary to know the total number of fixed points of a rational
function. The following lemma gives the total number of fixed point that a rational function can have.
 ${}$\\ ${}$\\
\begin{lemma}
    
\cite{Iteration of rational functions: Complex analytic dynamical systems} A rational function $R$ having degree $d\geq 1$ has at most $d+1$ fixed points (with counting multiplicity).
\end{lemma}

 There is an important quantity associated with a fixed point which is called residue fixed point index. This residue fixed point index correlates the multiplier of a fixed point and hence the nature of the fixed point. 
 
\begin{definition}{(Residue fixed point index )}\cite{Milnor}
Let $R$ be a rational function, the the residue fixed point index of the rational function $R$ at a fixed point $\hat{z}$, designated by $\imath(R,\hat{z})$ and it is defined as,
\begin{align*}
\dfrac{1}{2\pi i}\oint_{\gamma}\dfrac{dz}{z-R(z)}
\end{align*} 
where $\gamma$ is a small simple loop in the positive direction around the fixed point $\hat{z}$ such that it does not enclose any other fixed point of the function $R(z)$.\\
The word  mentioned above ``small simple'' can be any circle centered at the fixed point $\hat{z}$ with radius $0<r<\displaystyle\min_{R(z)=z}\vert z-\hat{z}\vert$.\\
\end{definition} 

\begin{definition}
    
Two rational maps $R$ and $S$ are said to be conjugate ($R \backsim S$) if and only if there is some M\"obius map $g$ with the following  property.
\begin{center}
$S=gRg^{-1}$
\end{center}
\end{definition}

\begin{lemma}\cite{The real non-attractive fixed point conjecture and beyond} If $R$ is a rational function such that the multiplier $\lambda$ of a fixed point is not equal to $1$ then the residue fixed point index of the rational function at the fixed point is given by $\dfrac{1}{1-\lambda}$.
\end{lemma}

\begin{remark}
    
If the multiplier of a fixed point $\hat{z}$ is $1$ i.e., $\lambda=R^{\prime}(\hat{z})=1$, the residue fixed point index $\imath(R,\hat{z})$ is still well defined and also it is finite \cite[Problem 12-a]{Milnor}.\\
\end{remark}

If two rational functions $R$ and $S$ are conjugate. Then the following lemma proves that the residue fixed point index is invariant under conformal conjugacy.

\begin{lemma}
\cite{The real non-attractive fixed point conjecture and beyond}
Let the rational functions $R$ and $S$ be conformally conjugate, then there exits a M\"obius map $g$ such that $S(z)=g(R(g^{-1}))(z)$ for any $z\in\widehat{\mathbb{C}}$. If the rational function $R$ has a fixed point $\hat{z}$ then $g(\hat{z})$ is a fixed point of $S$ and both the functions $R$ and $S$ have same residue fixed point indices at these fixed point, i.e., $\imath(R,\hat{z})=\imath(S,\hat{z})$.
\end{lemma}
 
The next theorem is known as \textbf{The rational fixed point theorem \cite{The real non-attractive fixed point conjecture and beyond},} it is also known as the Holomorphic fixed point formula.
 
\begin{thm}\cite{The real non-attractive fixed point conjecture and beyond} Let $R$ be any non-identity and non-constant rational map, then the sum of all residue fixed point indices over the extended complex plane $\widehat{\mathbb{C}}$ is equal to one, i.e., $\displaystyle\sum _{z=R(z)}\imath(R,{z})=1$.
\end{thm}

\begin{remark}\cite{The real non-attractive fixed point conjecture and beyond}
    
Since for each polynomial the point at infinity is super attracting fixed point and therefore it is simple fixed point, then residue fixed point index of each polynomial $P$, with degree at least $2$ at the point $\infty$ is $1$, therefore we have 
\begin{align*}
\displaystyle\sum _{z=R(z);z\in\mathbb{C}}\imath(P,{z})&=0.
\end{align*}
We already have a classification of fixed points with respect to the multipliers, next lemma also classifies the nature of the fixed points on the basis of residue fixed point index.
\end{remark}

 ${}$\\ ${}$\\

The following theorems resolves the real non attractive fixed point conjecture in it's full generality.

\begin{thm}{(The real non-attractive fixed point conjecture)}\cite{The real non-attractive fixed point conjecture and beyond}\label{T1}
Every rational function $R$ with degree at least two and with a super attracting fixed point has a fixed point such that the real part of its multiplier is greater than or equal to $1$. In particular, this is true for all polynomials with degree greater than or equal to two.
\end{thm}
Now we are going to prove this conjecture for complex harmonic functions. A complex-valued function in a planar domain is said to be harmonic if its real and imaginary parts are harmonic functions, not necessarily harmonic conjugate. \\\\
A continuous function $f = u + iv$ defined in a domain $D\subset \mathbb{C}$ is harmonic in $D$ if u and v are real harmonic functions in $D$ which are not necessarily conjugate. If $D$ is simply connected domain we can write $f={h}+\overline{g}$, where $h$ and $g$ are analytic functions on $D$ and $g$ denote the function $z\rightarrow\overline{g(z)}$.

\begin{definition}[$\mathfrak{h}$-Fixed Points]\cite{Some questions about complex harmonic functions}
The complex $\zeta$ is said to be a finite $\mathfrak{h}$-fixed point for the complex harmonic function $f=h+\overline{g}$ if $\zeta=\mu+\overline{\omega}$ and satisfies the equation $\mu+\overline{\omega}=h(\mu)+\overline{g(\omega)}$.

\end{definition}
Suppose that $f=h+\overline{g}$ such that $h(\mu)=\mu$ and $g(\omega)=\omega$, then $\mu+\overline{\omega}=h(\mu)+\overline{g(\omega)}$. Thus $\zeta=\mu+\overline{\omega}$ is a $\mathfrak{h}$- fixed point of $f$. These types of points will be called induced $\mathfrak{h}$-fixed points and they constitute probably isolated $\mathfrak{h}$-fixed points. In particular, if $g(0)=0$ then all the usual fixed points of $h$ are $\mathfrak{h}$ fixed points
 of $f=h+\overline{g}$. In the same way, if $h(0)=0$ then each fixed point of $g$ is $\mathfrak{h}$-fixed point of $f=h+\overline{g}$. It shows that 0 is a $\mathfrak{h}$-fixed point of  $f=h+\overline{g}$ whenever $h(0)=g(0)=0$.

${}$\\ ${}$\\

 In the case $D={\mathbb{C}}$ is reasonable to include $\zeta=\infty$ in the analysis of the $\mathfrak{h}$-fixed points of a complex harmonic functions. We say that
\begin{enumerate}
    \item[1.] $\mu+\overline{\infty}$ is an infinite $\mathfrak{h}$-fixed point of $f=h+\overline{g}$ if $h(\mu)=\mu$ and $g(\infty)=\infty$,
    \item[2.] $\infty+\overline{\omega}$ is an infinite $\mathfrak{h}$-fixed point of $f=h+\overline{g}$ if $h(\infty)=\infty$ and $g(\omega)=\omega$,
    \item[3.] $\infty+\overline{\infty}$ is an infinite $\mathfrak{h}$-fixed point of $f=h+\overline{g}$ if $h(\infty)=\infty$ and $g(\infty)=\infty$,
\end{enumerate}

\begin{remark}\cite{Some questions about complex harmonic functions}
    
 Let $\mu+\overline{\omega}$ be a $\mathfrak{h}$-fixed point of $f=h+\overline{g}$ then $\lambda=\partial_{z}h(\mu)=h'(\mu)=$ and $\theta=\partial_{z}g(\omega)=g'(\omega)$ are called the multipliers of $f$.

\end{remark}
\begin{example}
For $f(z) = z^2 + \overline{z^2}$, if $\mu = 1$, $\omega = 1$, then $h(1) = 1$, $\overline{g(1)} = 1$, so $\zeta = 1 + 1 = 2$ is an $\mathfrak{h}$-fixed point.
\end{example}
\begin{figure}[h]
    \centering
    \includegraphics[width=0.5\textwidth]{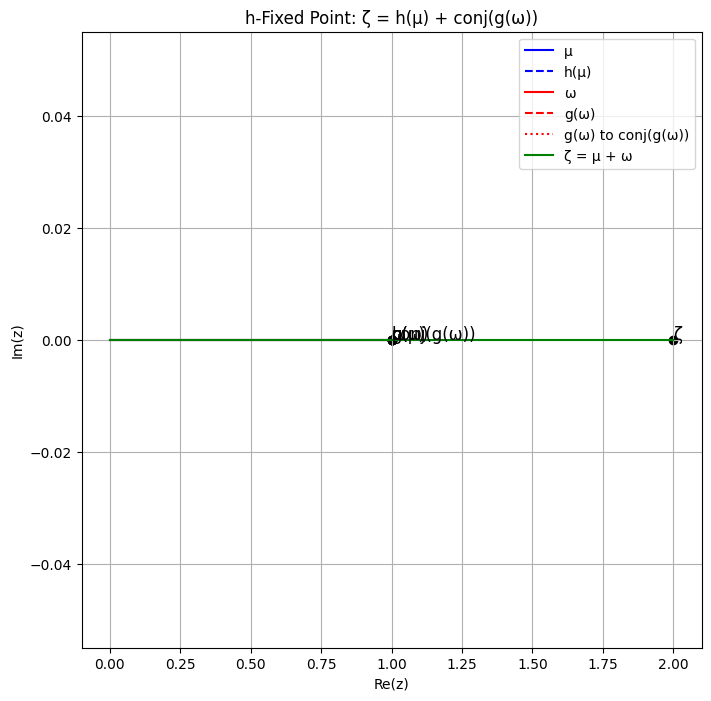}\label{f11}
    
 \caption{The $\mathfrak{h}$-fixed point $\mu+\overline{\omega}$ satisfies $\mu+\overline{\omega}=h(\mu)+\overline{g(\omega)}$ for $f(z) = z^2 + \overline{z^2}$. Arrows depict the mappings $\mu \to h(\mu)$ and $\omega \to g(\omega) \to \overline{g(\omega)}$.}
 \end{figure}
  Local dynamics near a $\mathfrak{h}$-fixed point $\zeta=\mu+\overline{\omega}$ for the function $f(z)=z^2 +\overline{z^{2}}$ with $\mu=1$ and $\omega=1$, $\lambda=h^{\prime}(1)=2$ and $\theta=g^{\prime}(1)=2$(repelling). Left:$h(z)$ near $\mu$ Right:$g(z)$ near $\omega$. Red arrows indicate repelling behaviour $\left( \vert\lambda\vert, \vert\theta\vert>1\right) 
$
\begin{figure}[h]
\begin{center}
\includegraphics[keepaspectratio=true,scale=0.2]{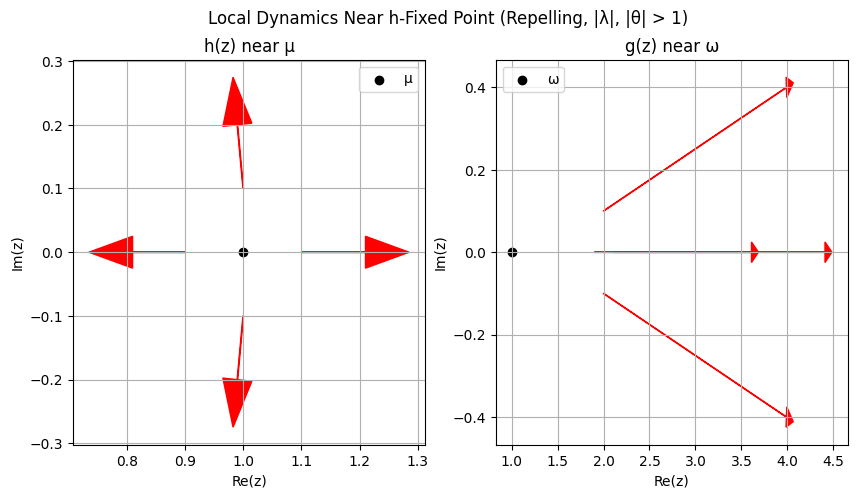}
\caption{Local dynamics near a $\mathfrak{h}$-fixed point $\zeta=\mu+\overline{\omega}$ for the function $f(z)=z^2 +\overline{z^{2}}$ }
\end{center}
\end{figure}

\begin{remark}\cite{Dynamics of Complex Harmonic Mappings}
A harmonic mapping $f=h+\overline{g}$ is said to be a polynomial (rational) harmonic mapping if both $h$ and $g$ are polynomials (rationals) of degrees atleast 2, and a transcendental harmonic mapping if at least one of the $h$ and $g$ is a transcendental function.
    
\end{remark}

We already have  real non attractive fixed point conjecture for polynomials and rational functions \cite{The real non-attractive fixed point conjecture and beyond}. Now we are going to prove this conjecture for polynomial harmonic functions and rational harmonic functions.

\begin{thm}\label{T1}

Every rational harmonic function $f=h+\overline{g}$ where both the rational functions $h$ and $g$ are of degree greater than or equal to two and with a super-attracting fixed point, has a $\mathfrak{h}$-fixed point  $\zeta=\mu+\overline{\omega}$ such that the real part of its multipliers is greater than or equal to $1$, i.e., $Re\left( \lambda=\partial_{z}h(\mu)\right)\geq 1$ and $Re\left( \theta=\partial_{z}g(\mu)\right)\geq 1$. In particular, this is true for all polynomial harmonic functions with degree greater than or equal to two.
    
\end{thm}

\begin{proof}

Let $f=h+\overline{g}$ be rational harmonic function, where both the rational functions, $h$ and $g$ are of degree greater than or equal to two and with a super-attracting fixed point.

Now firstly we are trying to show for the rational function $h$ that there exist a fixed point whose real part is greater than equal to one. If a rational function $h$ has a multiple fixed point then in this case it's multiplier is $1$ and real part of its multiplier is one.\\

Now consider the case when the rational function $h$ has all of its fixed points simple (Multipliers different from one). If rational function $h$ is of degree $d$ then it has $d+1$ distinct fixed points. Suppose that the multiplier $\lambda_{i}=\partial_{z}h(\mu_{i})$ for $i=1,2,3,\ldots,d+1$, where $h(\mu_{i})=\mu_{i}$ for $i=1,2,3,\ldots, d+1.$ 

Now by  using the rational fixed point theorem we have, $\displaystyle\sum_{i=1}^{d+1}\frac{1}{1-\lambda_{i}}=1$. However, if the rational function $h$ has a super attracting fixed point, say $\mu_{d+1}$, for this point its multiplier is $\lambda_{d+1}=0$. Therefore,
\begin{equation}\label{E3}
\displaystyle\sum_{i=1}^{d}\frac{1}{1-\lambda_{i}}=0
\end{equation}
Comparing the real part, we get that
\begin{align*}
Re\left( \frac{1}{1-\lambda_{i}}\right)&=\frac{1}{2}\left[ \frac{1}{1-\lambda_{i}}+\frac{1}{1-\bar{\lambda_{i}}}\right]\quad\quad\quad(\lambda_{i}=x+iy) \\
&=\frac{1-\bar{\lambda_{i}}+1-\lambda_{i}}{2(1-\lambda_{i})\overline{(1-\lambda_{i})}}\\
&=\frac{1-x+iy+1-x-iy}{2{\vert 1-\lambda_{i}\vert}^{2}}\\
&=\frac{1-x}{{\vert 1-\lambda_{i}\vert}^{2}}\\
&=\frac{Re(1-\bar{\lambda_{i}})}{{\vert 1-\lambda_{i}\vert}^{2}}
\end{align*}
Hence by equation \ref{E3}

\begin{equation}\label{E4}
\displaystyle\sum_{i=1}^{d}\frac{Re(1-\bar{\lambda_{i}})}{\mid 1-\lambda_{i}
\mid^{2}}=0 
\end{equation}
Since $Re(1-\bar{\lambda_{i}})=1-Re(\lambda_{i})$ and $\mid 1-\lambda_{i}\mid^{2}>0$ for each $i$, it is not possible to have $Re(\lambda_{i})<1$ for every $i$. Therefore, there exists $j\in\lbrace1,2,\cdots,d\rbrace$ such that $Re\left( \lambda_{i}=\partial_{z}h(\mu)\right)\geq 1$.

Therefore every rational function $h$ having degree at least two and with a super attracting fixed point has a fixed point, the real part of whose multiplier is greater than or equal to $1$.

Similarly we can show the existence of a fixed of rational function $g$ whose real part of multiplier is greater than or equal to one, $Re\left( \theta_{i}=\partial_{z}g(\mu_{i})\right)\geq 1$.

 Therefore, both the rational functions $h$ and $g$ have fixed points, say $\mu$ and $\omega$ such that $h(\mu)=\mu$ and $g(\omega)=\omega$. Then the real part of the multipliers  $Re\left( \lambda=\partial_{z}h(\mu)\right)\geq 1$ and $Re\left( \theta=\partial_{z}g(\mu)\right)\geq 1$. Also as $\mu+\overline{\omega}=h(\mu)+\overline{g(\omega)}$. Thus $\mu+\overline{\omega}$ is a $\mathfrak{h}$- fixed point of $f$. Therefore,   Every rational harmonic function $f=h+\overline{g}$, where both the rational functions $h$ and $g$ are of degree greater than or equal to two and with a super-attracting fixed point, has a $\mathfrak{h}$-fixed point  $\zeta=\mu+\overline{\omega}$ such that the real part of its multipliers is greater than or equal to $1$.
\end{proof}

\begin{remark}
\begin{enumerate}

\item[1.] If the real part of each multipliers is greater than $1$ i.e., $Re\left( \lambda=\partial_{z}h(\mu)\right)>1$ and $Re\left( \theta=\partial_{z}h(\mu)\right)>1$. then each term in the left hand side of Equation \ref{E4} is negative, which gives a contradiction. Therefore, every rational harmonic function $f$ with a super-attracting fixed point or every polynomial harmonic function having degree at least two has a $\mathfrak{h}$-fixed point such that the real part of its multipliers is less than or equal to $1$.\\
Furthermore, if all the $\mathfrak{h}$ fixed points of a rational harmonic function are simple and some fixed points have real part of their multipliers greater than $1$ then there is a $\mathfrak{h}$ fixed point having real part of its multipliers less than $1$.

\item[2.] Equating the imaginary part in Equation \ref{E3} for each rational function, we have,
\begin{center}
$\displaystyle\sum_{i=1}^{d}\frac{Im(\lambda_{i})}{\mid 1-\lambda_{i}\mid^{2}}=0$
\end{center}
 where $\lambda_{i}=\partial_{z}h(\mu_{i})$.\\
 
 From this equation, we observe that the rational function $h$ has a fixed point with multiplier 1 or has a fixed point with imaginary part at least $0$.\\
In other words, if all the fixed points of a rational function are simple then it has a fixed point such that the imaginary part of its multiplier is non-negative.
Hence it shows that if all the fixed point of a rational function are simple then it has a fixed point such that the imaginary part of its multiplier is non-negative.
\item[3.] The real non attractive fixed point conjecture is true for all rational harmonic functions $f$ with a super attracting fixed point.
However, we cannot say anything in general for a rational map without any super attracting fixed point.
\end{enumerate}
This can be seen through the following examples.
\end{remark}

\begin{thm}\label{T2}
    Every polynomial harmonic function $f=h+\overline{g}$, where both the polynomials $h$ and $g$ are of degree greater than or equal to two, has a $\mathfrak{h}$-fixed point, $\zeta=\mu+\overline{\omega}$ such that the real part of its multipliers is greater than or equal to 1, i.e., $Re\left( \lambda=\partial_{z}h(\mu)\right)\geq 1$ and $Re\left( \theta=\partial_{z}g(\mu)\right)\geq 1$.
\end{thm}
\begin{proof}
Let $f=h+\overline{g}$ be a polynomial harmonic function, where both the polynomials $h$ and $g$ are of degree greater than or equal to two. As we know $\infty$ is super attracting fixed point for a polynomial function and its multiplier is zero.

Therefore, by Theorem \ref{T1}, both the polynomial $h$ and $g$ has fixed points, say $\mu$
 and $\omega$ such that $h(\mu)=\mu$ and $g(\omega)=\omega$. Then the real part of the multipliers  $Re\left( \lambda=\partial_{z}h(\mu)\right)\geq 1$ and $Re\left( \theta=\partial_{z}g(\mu)\right)\geq 1$. Also as $\mu+\overline{\omega}=h(\mu)+\overline{g(\omega)}$. Thus $\mu+\overline{\omega}$ is a $\mathfrak{h}$- fixed point of $f$. Therefore,   Every polynomial harmonic function $f=h+\overline{g}$, where both the polynomials $h$ and $g$ are of degree greater than or equal to two, has a $\mathfrak{h}$-fixed point, $\zeta=\mu+\overline{\omega}$ such that the real part of its multipliers is greater than or equal to 1
\end{proof}

\begin{example}[Cubic Harmonic Function]
\label{cubic}
Consider $f(z)=z^3+\overline{z^3}$, with $h(z)=z^3$, $g(z)=z^3$. Fixed points of $h$: $z^3 = z$, so $\mu = 0, 1, -1$. Similarly, $\omega = 0, 1, -1$. are fixed points of $g$. Now $\mathfrak{h}$- fixed points are $\zeta = \mu + \omega$:
\begin{itemize}
    \item $\mu = 1$, $\omega=1$, $\zeta = 2$, $h(1) = 1$, $\overline{g(1)} = 1$.
    \item Multipliers: $\lambda = h^{\prime}(1)=3 \cdot 1^{2}= 3$, and $\theta = g'(1) = 3$, $\text{Re}(\lambda) = \text{Re}(\theta) = 3 \geq 1$.
    \item $\mu = -1$, $\omega = -1$, $\zeta = -2$, same multipliers.
    \item $\mu = 1$, $\omega = -1$, $\zeta = 0$, $\lambda = \theta = 3$.
    \item $\mu = 0$, $\omega = 0$, $\zeta = 0$, $\lambda = \theta = 0$ (super-attracting).
    \item $\zeta = 1$, $\zeta = -1$ have mixed multipliers (e.g., $\lambda = 0$, $\theta = 3$).
\end{itemize}

\begin{figure}[h]
    \centering
    \includegraphics[width=0.2\textwidth]{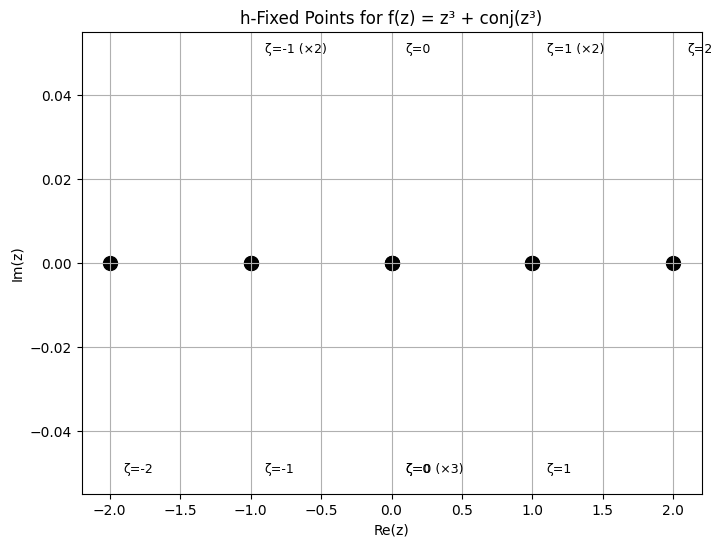}
    \caption{$\mathfrak{h}$-Fixed Points for $f(z) = z^3 + \overline{z^3}$. Points $\zeta = 2, -2$ have multipliers $\lambda = \theta = 3$.}
    \label{fig:cubic}
\end{figure}

\end{example}
\begin{example}
    \cite{The real non-attractive fixed point conjecture and beyond} Consider the rational function $R(z)=\dfrac{1}{z^{n-1}}$ where $n>1$.\\
Fixed points of $R(z)$ are,
\begin{align*}
R(z)&=z\quad\quad \mbox{also} \quad R^{\prime}(z)=(1-n)z^{-n}\\
\frac{1}{z^{n-1}}&=z\\
z^{d}&=1\\
z&=(1)^{\frac{1}{n}}\\
\mbox{$d$-th root of unity}\\
w_{k}&=e^{\frac{2k\pi i}{n}}, \quad\quad k=0,1,2,3\cdots,n-1.
\end{align*} Multiplier of each fixed point $w_{k}$ of this rational function is $(1-n)$ and it will be is negative for $n>2$. Which gives that the real non attractive fixed point conjecture conjecture is false for $\dfrac{1}{z^{n-1}}$, $n>2$.
\end{example}
 
\begin{example}\cite{The real non-attractive fixed point conjecture and beyond} Consider the rational function $f(z)=h(z)+\overline{g(z)}=\dfrac{cz}{z^{2}+z+1}+\overline{\dfrac{cz}{z^{2}+z+1}}$, where $c\neq 0$. Here $0$ is $\mathfrak{h}$-fixed point of $f(z)$.\\
 Also $h^{\prime}(z)=\dfrac{c(1-z^{2})}{(z^2+z+1)^2}$ and $g^{\prime}(z)=\dfrac{c(1-z^{2})}{(z^2+z+1)^2}$  which gives that $h^{\prime}(0)=c$ and $g^{\prime}(0)=c$. Thus the real non attractive fixed point conjecture is true for$f(z)=\dfrac{cz}{z^{2}+z+1}+\overline{\dfrac{cz}{z^{2}+z+1}}$ whenever $Re(c)\geq 1$.

  \end{example}

 \subsection{Beyond the real non-attractive fixed point conjecture}
Now the extreme situation of the real non attractive fixed point conjecture, that is, when the real part of the multipliers of the fixed point is exactly equal to one is discussed.
\subsection*{Multipliers with real part one}
So far the real non attractive fixed point conjecture is discussed which says every rational harmonic function having degree at least two and with a super attracting fixed point or a polynomial harmonic function having degree at least $2$ has a $\mathfrak{h}$-fixed point such that the real part of multiplier of this fixed point is greater than or equal to $1$.

Now the problem is to find out the condition, when will the multipliers of all the finite $\mathfrak{h}$-fixed points of a polynomial harmonic function have real part exactly equal to $1$? This is an supreme situation of the real non attractive fixed point conjecture.
${}$\\ ${}$\\
 	As we know that every quadratic polynomial is conformally conjugate to $z^{2}+c$ for some $c$ and also the multipliers of the fixed points are preserved under conformal conjugacy, so it is enough to consider the quadratic family $z^{2}+c$ \cite{The real non-attractive fixed point conjecture and beyond}.
 	\begin{thm}\label{T3}
 	Let $P(z)=f+\overline{f}=z^{2}+c+\overline{z^{2}+c}$ be the quadratic family, if $c=\frac{1}{4}$ then the polynomial harmonic function $P$ has a single $\mathfrak{h}$-fixed point and its multipliers are $1$. If $c\neq\frac{1}{4}$ then the polynomial harmonic function  $P(z)=z^{2}+c+\overline{z^{2}+c}$ has two simple $\mathfrak{h}$-fixed points. Further, the multipliers of the $\mathfrak{h}$-fixed points have real part equal to $1$ if and only if $c>\frac{1}{4}$. In other words, the multipliers of each $\mathfrak{h}$-fixed point of $P$ has real part equal to $1$ if and only if $c\geq\frac{1}{4}$.
 	
 	\end{thm}
 	\begin{proof}
 	To find the finite $\mathfrak{h}$-fixed points of the polynomial harmonic function $P(z)$ we put it in the following form,
\begin{align*}
f(z)&=z\quad\quad \mbox{Also}\quad f^{\prime}(z)=2z\\
z^{2}+c&=z\\
z^{2}-z+c&=0\\
z&=\frac{1\pm\sqrt{1-4c}}{2}
\end{align*}
So the finite fixed points of $f(z)$ are $\dfrac{1\pm\sqrt{1-4c}}{2}$ with multipliers $1\pm\sqrt{1-4c}$.\\
 If $c=\frac{1}{4}$, then $z=\frac{1}{2}$, i.e., there will be only one fixed point with multiplier $1$, as we have $f^{\prime}(z)=2z$, $f^{\prime}(\frac{1}{2})=1$.\\
 Hence there is only one $\mathfrak{h}$ fixed point if and only if $c=\frac{1}{4}$ and in this case multipliers is $1$.
 ${}$\\ ${}$\\
 If $c\neq\frac{1}{4}$, then there are two distinct finite fixed points of $f$ with multipliers $1+\sqrt{1-4c}$ and $1-\sqrt{1-4c}$ respectively. The real part of these multipliers is exactly equal to $1$ if and only if $1-4c<0$ which is nothing but $c>\frac{1}{4}$.\\
 Hence, the necessary and sufficient condition for the multipliers of each $\mathfrak{h}$-fixed point of $P$  to have real part exactly equal to $1$ is that $c\geq\frac{1}{4}$.
 	\end{proof}
\section{Problems}

\begin{itemize}
    \item \textbf{Transcendental Harmonic Functions}: Does the conjecture hold for $f = h + \overline{g}$, where $h$ or $g$ is transcendental (e.g., $h(z) = e^z$, $g(z) = z^2$)? For $f(z) = e^z + \overline{z^2}$, no finite h-fixed points exist.
 
\end{itemize}
\section{Acknowledgement}
The author acknowledges the inspiration drawn from the work of Rajen Kumar and Tarakanta Nayak \cite{The real non-attractive fixed point conjecture and beyond} in developing the harmonic analogues of the real non-attractive fixed point conjecture.

\end{document}